\documentclass[12pt,reqno]{amsart}
\usepackage{amsmath,amsfonts,amsthm,amssymb,color,mathrsfs}

\usepackage{hyperref}
\usepackage[T1]{fontenc}

\newcommand{\ou}{[0,1]}
\definecolor{dg}{rgb}{0, 0.5, 0}

\newcommand{\bd}{\mathbf{D}}
\newcommand{\bj}{\mathbf{J}}

\newcommand{\id}{\mbox{Id}}

\newcommand{\lp}{\left(}
\newcommand{\rp}{\right)}
\newcommand{\lc}{\left[}
\newcommand{\rc}{\right]}
\newcommand{\lcl}{\left\{}
\newcommand{\rcl}{\right\}}
\newcommand{\lln}{\left|}
\newcommand{\rrn}{\right|}
\newcommand{\lla}{\left\langle}
\newcommand{\rra}{\right\rangle}

\newcommand{\eps}{\varepsilon}

\newcommand{\ga}{\gamma}

\newcommand{\la}{\lambda}

\newcommand{\oom}{\Omega}

\newcommand{\beq}{\begin{equation}}
\newcommand{\eeq}{\end{equation}}
\newcommand{\bea}{\begin{eqnarray}}
\newcommand{\eea}{\end{eqnarray}}
\newcommand{\beas}{\begin{eqnarray*}}
\newcommand{\eeas}{\end{eqnarray*}}

\def\msh{{\mathscr H}}

\def\cC{{\mathcal C}}

\def\me{{\mathbb  E}}

\def\md{{\mathbb D}}
\def\mr{{\mathbb  R}}

\def\mp{{\mathbb  P}}

\newcommand{\cac}{{\mathcal C}}

\newcommand{\cf}{{\mathcal F}}

\newcommand{\ch}{{\mathcal H}}

\newcommand{\cm}{{\mathcal M}}

\newcommand{\crr}{{\mathcal R}}

\newcommand{\D}{{\mathbb D}}
\newcommand{\EE}{{\mathbb E}}

\newcommand{\PP}{{\mathbb P}}

\newcommand{\R}{{\mathbb R}}

\newtheorem{theorem}{Theorem}[section]

\newtheorem{corollary}[theorem]{Corollary}

\newtheorem{definition}[theorem]{Definition}

\newtheorem{hypothesis}[theorem]{Hypothesis}
\newtheorem{lemma}[theorem]{Lemma}

\newtheorem{proposition}[theorem]{Proposition}
\theoremstyle{remark}
\newtheorem{remark}[theorem]{Remark}
\theoremstyle{remark}
\newtheorem{example}[theorem]{Example}
\theoremstyle{remark}

\newtheorem{foo}[theorem]{Remarks}

\title[Varadhan estimates for RDEs driven by fBms]{Varadhan Estimates for rough differential equations driven by fractional Brownian motions}

\author[F. Baudoin  \and C. Ouyang \and X. Zhang]{Fabrice Baudoin,  Cheng Ouyang and  Xuejing Zhang}

\address{Fabrice Baudoin, Dept. Mathematics, Purdue University, 150 N. University St., West Lafayette, IN 47907-2067, USA.}
\email{fbaudoin@math.purdue.edu}

\address{Cheng Ouyang, Dept. Mathematics, Statistics and Computer Science, University of Illinois at Chicago, 851 S. Morgan St., Chicago, IL 60607, USA.}
\email{couyang@math.uic.edu}

\address{Xuejing Zhang, Dept. Mathematics, Purdue University, 150 N. University St., West Lafayette, IN 47907-2067, USA.}
\email{zhang239@math.purdue.edu}

\begin{document}
\maketitle

\begin{abstract}
In this work we study rough differential equations driven by a fractional Brownian motion with Hurst parameter $H>\frac{1}{4}$ and establish Varadhan's small time estimates for the density of solutions of  such equations. under H\"ormander's type conditions \end{abstract}

\tableofcontents

\section{Introduction}
Let $B=(B^1,...,B^d)$ be a $d$-dimensional fractional Brownian motion with Hurst parameter $H>\frac{1}{4}$, that is, $B$ is a $\mr^d$-valued centered Gaussian process with covariance
$$\me (B^i_tB^j_s)=\frac{1}{2}(t^{2H}+s^{2H}-|t-s|^{2H})\delta_{ij}.$$
A straightforward application of the Kolmogorov continuity theorem shows that the Hurst parameter $H$ controls the path regularity of $B$:  the sample paths of $B$ are  almost surely locally $\gamma$-H\"{o}lder continuous for all $\gamma<H$. When $H=\frac{1}{2}$, $B$ is a standard Brownian motion.

We are interested here in the following family of stochastic differential equations driven by $B$:
\begin{align}\label{SDE-intro}
X_t^\eps=x+\eps\sum_{i=1}^d\int_0^tV_i(X_s^\eps)dB_s^i,\quad\quad\varepsilon\in(0,1),
\end{align}
where  $V_1,\ldots,V_d$ are $C^\infty$-bounded vector fields on $\R^n$. When $H>\frac{1}{2}$, the above equation is understood in the sense of Young's integration. In that case, existence and uniqueness of solutions are well-established, for instance, in \cite{NR} and \cite{Za}. When $\frac{1}{4}<H\leq\frac{1}{2}$, equation (\ref{SDE-intro}) is interpreted in  the framework of rough path theory (see \cite{FV-bk, LQ}). Existence and uniqueness of solutions in this case can be found, for example, in \cite{LQ}. In particular, when $H=\frac{1}{2}$, this notion of solution coincides with the solution of the corresponding Stratonovitch stochastic differential equation. 

Once equation (\ref{SDE-intro}) is properly interpreted and solved with a unique solution, a natural question to address and most relevant to our purpose is the existence of a (smooth) density function of the random variable $X^\eps_t$. In the regular case $H>\frac{1}{2}$, the authors proved (see \cite{BH}) that under the classical H\"{o}rmander's condition on the vector fields, the law of $X^\eps_t$ admits a smooth density with respect to the Lebesgue measure. The existence and smoothness of density function is more involved in the rough case.  Still under H\"{o}rmander condition, the existence of a density function for $\frac{1}{4}<H<\frac{1}{2}$ is due to  \cite{CF}. The smoothness of this density is proved in \cite{H-P} for $H>\frac{1}{3}$, conditioned on the integrability of the Jacobian of the system which is established later in \cite{CLL}. Finally, smoothness of the density for all $H>\frac{1}{4}$ is proved in \cite{CHLT}.

\

With the existence of the density in hands, the next step is to study some of its basic properties. Small time asymptotics in the regular case $H>1/2$ have been studied in \cite{BO}. Gaussian or sub-Gaussian upper bounds have been established in  \cite{BOT} and \cite{BNOT}. The positivity of the density is proved in \cite{BNOT}. In this work, we are interested in Varadhan type estimates for this density. 

\

Throughout our discussion, we will  assume some uniform hypoellipticity condition inspired by Kusuoka \cite{Ku} (see Hypothesis \ref{UH condition} for details). 
The main problem we are interested in is to establish a version of Varadhan's estimates for the density of $X_1^\eps$. More precisely,   introduce the following functions on $\mr^n$, 

$$d^2(y)=\inf_{\Phi_1(h)=y}\frac{1}{2}\|h\|_{\msh}^2,\quad\mathrm{and}\quad d^2_R(y)=\inf_{\Phi_1(h)=y, \det\gamma_{\Phi_1(h)}>0}\frac{1}{2}\|h\|_\msh^2,$$

where $\Phi$ is obtained by solving the ordinary diferential equation driven by Cameron-Martin paths:
\begin{align*}\Phi_t(h)=x+\sum_{i=1}^d\int_0^tV_i(\Phi_s(h))dh_s^i.\end{align*}

Our  main result is the following.
\begin{theorem}\label{th: main intro}
Let us denote by $p_\eps(y)$ the density of $X_1^\eps$. Then
\begin{align*}
\liminf_{\eps\downarrow0}\eps^2\log p_\eps(y)\geq -d^2_R(y),
\end{align*}
and
\begin{align*}
\limsup_{\eps\downarrow0}\eps^2\log p_\eps(y)\leq -d^2(y).
\end{align*}
Moreover, if $\inf_{\Phi_1(h)=y, \det\gamma_{\Phi_1(y)}>0}\det\gamma_{\Phi_1(h)}>0$, then
\begin{align*}
\lim_{\eps\downarrow0}\eps^2\log p_\eps(y)=-d^2_R(y).
\end{align*}
\end{theorem}

In the classical case when $B$ is a standard Brownian motion, these  results were studied by numerous authors, including L\'{e}andre \cite{Leandre}, Ben Arous and L\'{e}andre \cite{BL} and L\'{e}andre and Russo \cite{LR}.  Our result is obviously an extension of the classical result to the fractional Brownian motion case, in which many recent developments in rough paths theory are employed. In particular, we would like to mention the recent breakthrough \cite{CLL} in the integrality of the Jacobian of equation (\ref{SDE-intro}). This result  leads to the integrability of the Malliavin derivative $\bd X_1^\eps$ and inverse Malliavin matrix $\gamma_{X_1^\eps}$ of $X_1^\eps$. One of the main technical difficulties of this work is  to show that under the uniform hypoelliptic condition, for any fixed $r>0$
$$\|\gamma_{X_1^\eps}^{-1}\|_r\leq c_r\eps^{-2l},\quad\quad \eps\in(0,1],$$
for some constant $c_r$ depending on $r$ and constant $l$ depending on the structure of vector fields $V_i$'s which is given in Hypothesis \ref{UH condition}.

Finally, let us mention that some small-time asysmptotics of density function results have already been studied in the smooth case $H>\frac{1}{2}$ for elliptic systems in \cite{BO} and \cite{Inahama2}. These results clearly imply the Varadhan's asymptotics stated in Theorem \ref{th: main intro}. In the rough case $\frac{1}{4}<H<\frac{1}{2}$, the Laplace approximation is obtain by Inahama \cite{Inahama}, which is along the same line of research as the current paper.

The paper is organized as follows.  In section 2, we give the necessary background on rough path and Malliavin calculus that will be needed throughout the paper. We state and prove our main result in Section 3.

\section{Preliminary material}

For some fixed $H>\frac{1}{4}$, we consider $(\oom,\cf,\PP)$ the canonical probability space associated with the fractional
Brownian motion (in short fBm) with Hurst parameter $H$. That is,  $\oom=\cac_0([0,1])$ is the Banach space of continuous functions
vanishing at $0$ equipped with the supremum norm, $\cf$ is the Borel sigma-algebra and $\PP$ is the unique probability
measure on $\oom$ such that the canonical process $B=\{B_t=(B^1_t,\ldots,B^d_t), \; t\in [0,1]\}$ is a fractional Brownian motion with Hurst
parameter $H$.
In this context, let us recall that $B$ is a $d$-dimensional centered Gaussian process, whose covariance structure is induced by
\begin{align}\label{covariance}
R\left( t,s\right) :=\EE\lc  B_s^j \, B_t^j\rc
=\frac{1}{2}\left( s^{2H}+t^{2H}-|t-s|^{2H}\right),
\quad
s,t\in[0,1] \mbox{ and } j=1,\ldots,d.
\end{align}
In particular it can be shown, by a standard application of Kolmogorov's criterion, that $B$ admits a continuous version
whose paths are $\ga$-H\"older continuous for any $\ga<H$.

\subsection{Rough path}
For $N\in\mathbb{N}$, recall that the truncated algebra $T^{N}(\mathbb{R}%
^{d})$ is defined by
$$
T^{N}(\mathbb{R}^{d})=\bigoplus_{m=0}^{N}(\mathbb{R}%
^{d})^{\otimes m},
$$
with the convention $(\mathbb{R}^{d})^{\otimes
0}=\mathbb{R}$. The set $T^{N}(\mathbb{R}^{d})$ is equipped with a straightforward
vector space structure, plus an operation $\otimes$ defined by
\[
\pi_{m}(g\otimes h)=\sum_{k=0}^{N}\pi_{m-k}(g)\otimes\pi_{k}(h),\qquad g,h\in
T^{N}(\mathbb{R}^{d}),
\]
where $\pi_{m}$ designates the projection on the $m$th tensor level. Then
$(T^{N}(\mathbb{R}^{d}),+,\otimes)$ is an associative algebra with unit
element $\mathbf{1} \in(\mathbb{R}^{d})^{\otimes0}$.

\smallskip

For $s<t$ and $m\geq2$, consider the simplex $\Delta_{st}^{m}=\{(u_{1}%
,\ldots,u_{m})\in\lbrack s,t]^{m};\,u_{1}<\cdots<u_{m}\} $, while the
simplices over $[0,1]$ will be denoted by $\Delta^{m}$. A continuous map
$\mathbf{x}:\Delta^{2}\rightarrow T^{N}(\mathbb{R}^{d})$ is called a
multiplicative functional if for $s<u<t$ one has $\mathbf{x}_{s,t}%
=\mathbf{x}_{s,u}\otimes\mathbf{x}_{u,t}$. An important example arises from
considering paths $x$ with finite variation: for $0<s<t$ we set
\begin{equation}
\mathbf{x}_{s,t}^{m}=\sum_{1\leq i_{1},\ldots,i_{m}\leq d}\biggl( \int%
_{\Delta_{st}^{m}}dx^{i_{1}}\cdots dx^{i_{m}}\biggr) \,e_{i_{1}}\otimes
\cdots\otimes e_{i_{m}}, \label{eq:def-iterated-intg}%
\end{equation}
where $\{e_{1},\ldots,e_{d}\}$ denotes the canonical basis of $\mathbb{R}^{d}%
$, and then define the truncated \textit{signature} of $x$ as
\[
S_{N}(x):\Delta^{2}\rightarrow T^{N}(\mathbb{R}^{d}),\qquad(s,t)\mapsto
S_{N}(x)_{s,t}:=1+\sum_{m=1}^{N}\mathbf{x}_{s,t}^{m}.
\]
The function $S_{N}(x)$ for a smooth function $x$ will be our typical example of multiplicative functional. Let us stress the fact that those elements take values in the strict subset
$G^{N}(\mathbb{R}^{d})\subset T^{N}(\mathbb{R}^{d})$ given by the \emph{group-like}
elements
\[
G^{N}(\mathbb{R}^{d}) = \exp^{\oplus}\bigl(L^{N}(\mathbb{R}^{d})\bigr),
\]
where $L^{N}(\mathbb{R}^{d})$ is the linear span of all elements that can be
written as a commutator of the type $a \otimes b - b\otimes a$ for two
elements in $T^{N}(\mathbb{R}^{d})$. This set is called free nilpotent group of step $N$, and is equipped with the classical Carnot-Caratheodory norm which we simply denote by $|\cdot|$. For a path $\mathbf{x}\in\cC([0,1],G^{N}(\R^d))$, the $p$-variation norm of $\mathbf{x}$ is defined to be
\begin{align*}
\|\mathbf{x}\|_{p-{\rm var}; [0,1]}=\sup_{\Pi \subset [0,1]}\left(\sum_i |\mathbf{x}_{t_i}^{-1}\otimes \mathbf{x}_{t_{i+1}}|^p\right)^{1/p}
\end{align*}
where the supremum is taken over all subdivisions $\Pi$ of $[0,1]$.

\smallskip

With these notions in hand, let us briefly define what we mean by geometric rough path (we refer to \cite{FV-bk,LQ} for a complete overview): for $p\ge 1$, an element $x:\ou\to G^{\lfloor p \rfloor}(\R^d)$ is said to be a geometric rough path if it is the $p$-var limit of a sequence $S_{\lfloor p \rfloor}(x^{m})$ of lifts of smooth functions $x^m$. In particular, it is an element of the space
$$\cC^{p-{\rm var}; [0,1]}([0,1], G^{\lfloor p \rfloor}(\R^d))=\{\mathbf{x}\in \cC([0,1], G^{\lfloor p \rfloor}(\R^d)): \|\mathbf{x}\|_{p-{\rm var}; [0,1]}<\infty\}.
$$

Let $\mathbf{x}$ be a geometric $p$-rough path with its approximating sequence $x^m$, that is, $x^m$ is a sequence of smooth functions such that $\mathbf{x}^m=S_{\lfloor p\rfloor}(x^m)$ converges to $\mathbf{x}$ in the $p$-var norm. Fix any $1\leq q\leq p$ so that $p^{-1}+ q^{-1}>1$ and pick any $h\in\cC^{q-{\rm var}}([0,1], \R^d)$. One can define the translation of $\mathbf{x}$ by h, denoted by $T_h(\mathbf{x})$ by
$$T_h(\mathbf{x})=\lim_{n\to\infty}S_{\lfloor p\rfloor}({x}^m+h).$$
It can be shown that $T_h(\mathbf{x})$ is an element in $\cC^{p-{\rm var}}([0,1], G^{\lfloor p \rfloor}(\R^d))$. Moreover, one has the following continuity property.

\begin{lemma}\label{th: rough-couple-conti}
For any $1\leq q\leq p$ so that $p^{-1}+ q^{-1}>1$, let $(\mathbf{x}, h)\in \cC^{p-{\rm var}}([0,1], G^{\lfloor p \rfloor}(\R^d))\times \cC^{q-{\rm var}}([0,1], \R^d)$.  Denoted by $T_h(\mathbf{x})\in \cC^{p-{\rm var}}([0,1], G^{\lfloor p \rfloor}(\R^d))$the translation of $\mathbf{x}$ by $h$. We have
\begin{enumerate}
\item There is some constant $C$ depending only on $p$ and $q$,
$$\|T_h(\mathbf{x})\|_{p-{\rm var};[0,1]}\leq C(\|\mathbf{x}\|_{p-{\rm var};[0,1]}+\|h\|_{q-{\rm var};[0,1]}).$$
\item The rough path translation $(\mathbf{x}, h)\mapsto T_h(\mathbf{x})$ as a map from
$$\cC^{p-{\rm var}}([0,1], G^{\lfloor p \rfloor}(\R^d))\times \cC^{q-{\rm var}}([0,1], \R^d)\to \cC^{p-{\rm var}}([0,1], G^{\lfloor p \rfloor}(\R^d))$$
is uniformly continuous on bounded sets.
\end{enumerate}
\end{lemma}

\begin{remark}
A typical situation of the above translation of $\mathbf{x}$ by $h$ in the present paper is when $\mathbf{x}=\mathbf{B}$, the fractional Brownian motion lifted as a rough path, and $h$ is a Cameron-Martin element of $B$. In this case, we simply denote $T_h(\mathbf{B})=B+h.$ 
\end{remark}

According to the considerations above, in order to prove that a lift of a $d$-dimensional fBm as a geometric rough path exists it is sufficient to build enough iterated integrals of $B$ by a limiting procedure. Towards this aim, a lot of the information concerning $B$ is encoded in the rectangular increments of the covariance function $R$ (defined by \eqref{covariance}), which are given by
\begin{equation*}
R_{uv}^{st} \equiv \EE\lc (B_t^1-B_s^1) \, (B_v^1-B_u^1) \rc.
\end{equation*}
We then call 2-dimensional $\rho$-variation of $R$ the quantity
\begin{equation*}
V_{\rho}(R)^{\rho} \equiv
\sup \lcl
\lp \sum_{i,j} \lln R_{s_{i} s_{i+1}}^{t_{j}t_{j+1}} \rrn^{\rho} \rp^{1/\rho}; \, (s_i), (t_j)\in \Pi
\rcl,
\end{equation*}
where $\Pi$ stands again for the set of partitions of $\ou$. The following result is now well known for fractional Brownian motion:
\begin{proposition}\label{prop:fbm-rough-path}
For a fractional Brownian motion with Hurst parameter $H$, we have $V_{\rho}(R)<\infty$ for all $\rho\ge1/(2H)$. Consequently, for $H>1/4$ the process $B$ admits a lift $\mathbf{B}$ as a geometric rough path of order $p$ for any $p>1/H$.
\end{proposition}

\begin{proof}
The fact that $V_{\rho}(R)<\infty$ for all $\rho\ge1/(2H)$ is the content of \cite[Proposition 15.5]{FV-bk}. The implication on the rough path construction can also be found in \cite[Chapter 15]{FV-bk}.

\end{proof}

\subsection{Malliavin Calculus} We introduce the basic framework of Malliavin calculus in this subsection.  The reader is invited to read the corresponding chapters in  \cite{Nu06} for further details. Let $\mathcal{E}$ be the space of $\mathbb{R}^d$-valued step
functions on $[0,1]$, and $\mathcal{H}$  the closure of
$\mathcal{E}$ for the scalar product:
\[
\langle (\mathbf{1}_{[0,t_1]} , \cdots ,
\mathbf{1}_{[0,t_d]}),(\mathbf{1}_{[0,s_1]} , \cdots ,
\mathbf{1}_{[0,s_d]}) \rangle_{\mathcal{H}}=\sum_{i=1}^d
R(t_i,s_i).
\]
We denote by $K^*_H$ the isometry between $\mathcal{H}$ and $L^2([0,1])$.
When $H>\frac{1}{2}$ it can be shown that $\mathbf{L}^{1/H} ([0,1], \mathbb{R}^d)
\subset \mathcal{H}$, and when $\frac{1}{4}<H<\frac{1}{2}$ one has
$$C^\gamma\subset \mathcal{H}\subset L^2([0,1])$$
for all $\gamma>\frac{1}{2}-H$.

We remark that $\ch$ is the reproducing kernel Hilbert space for $B$. Let $\msh$ be the Cameron-Martin space of $B$,
one proves that the operator $\crr:=\crr_H :\ch \rightarrow \msh$ given by
\begin{equation}\label{eq:def-R}
\crr \psi := \int_0^\cdot K_H(\cdot,s) [K^*_H \psi](s)\, ds
\end{equation}
defines an isometry between $\ch$ and $\msh$. 
Let us now quote from \cite[Chapter 15]{FV-bk} a result relating the 2-d regularity of $R$ and the regularity of $\msh$.
\begin{proposition}\label{prop:imbed-bar-H}
Let $B$ be a fBm with Hurst parameter $\frac{1}{4}<H<\frac{1}{2}$. Then one has $\msh\subset \cac^{\rho-{\rm var}}$ for $\rho>(H+1/2)^{-1}$. Furthermore, the following quantitative bound holds:
\begin{equation*}
\Vert h \Vert_{\msh} \ge \frac{\Vert h \Vert_{\rho-{\rm var}}}{(V_\rho(R))^{1/2}}.
\end{equation*}
\end{proposition}
\begin{remark}
The above proposition shows that for fBm we have $\msh\subset \cac^{\rho-{\rm var}}$ for $\rho>(H+1/2)^{-1}$. Hence an integral of the form $\int h \, dB$ can be interpreted in the Young sense by means of $p$-variation techniques.
\end{remark}

A $\mathcal{F}$-measurable real
valued random variable $F$ is then said to be cylindrical if it can be
written, for a given $n\ge 1$, as
\begin{equation*}
F=f\lp  B(\phi^1),\ldots,B(\phi^n)\rp=
f \Bigl( \int_0^{1} \langle \phi^1_s, dB_s \rangle ,\ldots,\int_0^{1}
\langle \phi^n_s, dB_s \rangle \Bigr)\;,
\end{equation*}
where $\phi^i \in \mathcal{H}$ and $f:\mathbb{R}^n \rightarrow
\mathbb{R}$ is a $C^{\infty}$ bounded function with bounded derivatives. The set of
cylindrical random variables is denoted $\mathcal{S}$.

The Malliavin derivative is defined as follows: for $F \in \mathcal{S}$, the derivative of $F$ is the $\mathbb{R}^d$ valued
stochastic process $(\mathbf{D}_t F )_{0 \leq t \leq 1}$ given by
\[
\mathbf{D}_t F=\sum_{i=1}^{n} \phi^i (t) \frac{\partial f}{\partial
x_i} \left( B(\phi^1),\ldots,B(\phi^n)  \right).
\]
More generally, we can introduce iterated derivatives. If $F \in
\mathcal{S}$, we set
\[
\mathbf{D}^k_{t_1,\ldots,t_k} F = \mathbf{D}_{t_1}
\ldots\mathbf{D}_{t_k} F.
\]
For any $p \geq 1$, it can be checked that the operator $\mathbf{D}^k$ is closable from
$\mathcal{S}$ into $\mathbf{L}^p(\oom;\mathcal{H}^{\otimes k})$. We denote by
$\mathbb{D}^{k,p}$ the closure of the class of
cylindrical random variables with respect to the norm
\[
\left\| F\right\| _{k,p}=\left( \mathbb{E}\left( F^{p}\right)
+\sum_{j=1}^k \mathbb{E}\left( \left\| \mathbf{D}^j F\right\|
_{\mathcal{H}^{\otimes j}}^{p}\right) \right) ^{\frac{1}{p}},
\]
and
\[
\mathbb{D}^{\infty}=\bigcap_{p \geq 1} \bigcap_{k
\geq 1} \mathbb{D}^{k,p}.
\]

\begin{definition}\label{non-deg}
Let $F=(F^1,\ldots , F^n)$ be a random vector whose components are in $\mathbb{D}^\infty$. Define the Malliavin matrix of $F$ by
$$\gamma_F=(\langle \mathbf{D}F^i, \mathbf{D}F^j\rangle_{\ch})_{1\leq i,j\leq n}.$$
Then $F$ is called  {\it non-degenerate} if $\gamma_F$ is invertible $a.s.$ and
$$(\det \gamma_F)^{-1}\in \cap_{p\geq1}L^p(\Omega).$$
\end{definition}
\noindent
It is a classical result that the law of a non-degenerate random vector $F=(F^1, \ldots , F^n)$ admits a smooth density with respect to the Lebesgue measure on $\mr^n$. Furthermore, the following integration by parts formula allows to get more quantitative estimates:

\begin{proposition}\label{th: IBP}
Let $F=(F^1,...,F^n)$ be a non-degenerate random vector whose components are in $\D^\infty$, and $\gamma_F$ the Malliavin matrix of $F$. Let $G\in\D^\infty$ and $\varphi$ be a function in the space $C_p^\infty(\mr^n)$. Then for any multi-index $\alpha\in\{1,2,...,n\}^k, k\geq 1$, there exists an element $H_\alpha\in\D^\infty$ such that
$$\me[\partial_\alpha \varphi(F)G]=\me[\varphi(F)H_\alpha].$$
Moreover, the elements $H_\alpha$ are recursively given by
\begin{align*}
&H_{(i)}=\sum_{j=1}^{d}\delta\left(G(\gamma_F^{-1})^{ij}\mathbf{D}F^j\right)\\
&H_\alpha=H_{(\alpha_k)}(H_{(\alpha_1,..., \alpha_{k-1})}),
\end{align*}
and for $1\leq p<q<\infty$ we have
$$\|H_\alpha\|_{L^p}\leq C_{p,q}\|\gamma_F^{-1}\mathbf{D}F\|^k_{k, 2^{k-1}r}\|G\|_{k,q},$$
where $\frac{1}{p}=\frac{1}{q}+\frac{1}{r}$.
\end{proposition}
\begin{remark}\label{est H}
By the estimates for $H_\alpha$ above, one can conclude that there exist constants $\beta, \gamma>1$ and integers $m, r$ such that
\begin{align*}
\|H_\alpha\|_{L^p}\leq C_{p,q}\|\det\gamma^{-1}_F\|^m_{L^\beta}\|\mathbf{D}F\|_{k,\gamma}^r\|G\|_{k,q}.
\end{align*}
\end{remark}
\begin{remark}
In what follows, we use $H_\alpha(F,G)$ to emphasize its dependence on $F$ and $G$.
\end{remark}


\subsection{Differential equations driven by fractional Brownian motions}
Let $B$ be a d-dimensional fractional Brownian motion with Hurst parameter $H>\frac{1}{4}$. Fix a small parameter $\eps\in(0,1]$, and consider the solution $X_t^\eps$ to the stochastic differential equation
\begin{align}\label{equ: SDE}
X_t^\eps=x+\eps\sum_{i=1}^d\int_0^tV_i(X_s^\eps)dB_s^i,
\end{align}
where the vector fields $V_1,\ldots,V_d$ are $C^\infty$-bounded vector fields on $\R^n$. 


Proposition \ref{prop:fbm-rough-path} ensures the existence of a lift of $B$ as a geometrical rough path. The general rough paths theory (see e.g.~\cite{FV-bk,Gu}) allows thus to state  the following proposition:

\begin{proposition}\label{prop:moments-sdes-rough}
Consider equation (\ref{equ: SDE}) driven by a $d$-dimensional fBm $B$ with Hurst parameter $H>\frac{1}{4}$, and assume that the vector fields $V_i$s are $C^\infty$-bounded. Then

\smallskip

\noindent
\emph{(i)}
For each $\eps\in(0,1]$, equation (\ref{equ: SDE}) admits a unique finite $p$-var continuous solution $X^\eps$ in the rough paths sense, for any $p> \frac{1}{H}$.

\smallskip

\noindent
\emph{(ii)}
For any $\lambda>0$ and $\delta<\frac{1}{p}$ we have
\begin{equation}\label{eq:exp-delta-moments}
\me\left[\exp\lambda\left(\sup_{t\in [0,1], \epsilon \in (0,1]}|X^\eps_t|^\delta\right)\right]<\infty.
\end{equation}
\end{proposition}

Once equation (\ref{equ: SDE}) is solved, the vector $X_t^\eps$ is a typical example of  random variable which can be differentiated in the Malliavin sense. We shall express this Malliavin derivative in terms of the Jacobian $\bj^\eps$ of the equation, which is defined by the relation $$\bj_{t}^{\eps, ij}=\partial_{x_j}X_t^{\eps,i}.$$ Setting $DV_{j}$ for the Jacobian of $V_{j}$ seen as a function from $\R^{n}$ to $\R^{n}$, let us recall that $\bj^\eps$ is the unique solution to the linear equation
\begin{equation}\label{eq:jacobian}
\bj_{t}^\eps = \id_{n} +
\eps\sum_{j=1}^d \int_0^t DV_j (X^{\eps}_s) \, \bj_{s}^\eps \, dB^j_s,
\end{equation}
and that the following results hold true (see \cite{CF} and \cite{NS}  for further details):
\begin{proposition}\label{prop:deriv-sde}
Let $X^\eps$ be the solution to equation (\ref{equ: SDE}) and suppose the $V_i$'s are $C^\infty$-bounded. Then
for every $i=1,\ldots,n$, $t>0$, and $x \in \mathbb{R}^n$, we have $X_t^{\eps,i} \in
\mathbb{D}^{\infty}$ and
\begin{equation*}
\mathbf{D}^j_s X_t^{\eps}= \mathbf{J}^\eps_{st} V_j (X^\eps_s) , \quad j=1,\ldots,d, \quad
0\leq s \leq t,
\end{equation*}
where $\mathbf{D}^j_s X^{\eps,i}_t $ is the $j$-th component of
$\mathbf{D}_s X^{\eps,i}_t$, $\mathbf{J}_{t}^\eps=\partial_{x} X^\eps_t$ and $\bj_{st}^\eps=\bj_{t}^\eps(\bj_{s}^\eps)^{-1}$.

\end{proposition}

\smallskip

Let us now quote the recent result \cite{CLL}, which gives a useful estimate for moments of the Jacobian of rough differential equations driven by Gaussian processes.
\begin{proposition}\label{prop:moments-jacobian}
Consider a  fractional Brownian motion $B$ with Hurst parameter $H>\frac{1}{4}$ and $p>\frac{1}{H}$. Then for any $\eta\ge 1$, there exists a finite constant $c_\eta$ such that the Jacobian $\bj^\eps$ defined at Proposition \ref{prop:deriv-sde} satisfies:
\begin{equation}\label{eq:moments-J-pvar}
\EE\lc  \sup_{\eps\in[0,1]}\Vert \bj^\eps \Vert^{\eta}_{p-{\rm var}; [0,1]} \rc = c_\eta.
\end{equation}
\end{proposition}


In the sequel, we also need the following restatement of~\cite[Proposition 4.4]{H-P}.
\begin{proposition}\label{th:express DX}
\label{induction}
Fix  $k\in
\mathbb{N}
$ and let $\left\{ h_{1},\ldots,h_{k}\right\} $ be any family of elements in $\msh$. Then the directional derivative $\mathbf{D}_{h_{1}}\ldots \mathbf{D}_{h_{k}}X^\eps_t $ exists for any $t\in \left[
0,1\right] .$ Moreover, there exists a collection of finite indexing sets
\begin{equation*}
\left\{ \mathbf{K}_{\left( i_{1},\ldots ,i_{k}\right) }:\left(
i_{1},\ldots ,i_{k}\right) \in \left\{ 1,\ldots ,d\right\} ^{k}\right\},
\end{equation*}%
such that for every $j\in \left\{ 1,..,n\right\} $ we have
\begin{eqnarray}
&&\mathbf{D}_{h_{1}}\ldots \mathbf{D}_{h_{k}}X_{t}^{\eps,j}  \label{high order rep} \\
&=&\sum_{i_{1},\ldots ,i_{k}=1}^{d}\sum_{m\in \mathbf{K}_{\left(
i_{1},\ldots ,i_{k}\right) }}\int_{0<t_{1}<\cdots <t_{k}<t}f_{1}^{\eps,m}\left(
t_{1}\right) \ldots f_{k}^{\eps,m}\left( t_{m}\right) f_{k+1}^{\eps,m}\left( t\right)
d h_1^{i_1}(t_{1}) \ldots d h_k^{i_k}(t_{k}) ,  \notag
\end{eqnarray}%
for some functions $f_{\ell}^{\eps,m}$ which are in $\cac^{p-{\rm var}}\left( \left[
0,1\right] ,%
\mathbb{R}
\right) $ for every $p>\frac{1}{H}$, $\ell$ and $m,$ i.e.%
\begin{equation*}
\cup _{\left( i_{1},\ldots ,i_{k}\right) \in \left\{ 1,\ldots ,d\right\} ^{k}}\cup
_{m\in \mathbf{K}_{\left( i_{1},\ldots ,i_{k}\right) }}\left\{
f_{\ell}^{\eps,m}:\ell=1,..,k+1\right\} \subset \cac^{p-{\rm var}}\left( \left[ 0,1\right]
,%
\mathbb{R}
\right) .
\end{equation*}%
Furthermore, there exist constants $C>0$ and $\alpha\geq1$, which depend only on $m$
such that for all $t \in (0,1]$
\begin{equation}\label{eq:high-order-bound}
\left\Vert f_{\ell}^{\eps,m}\right\Vert _{p-{\rm var};\left[ 0,t\right] }\leq C\left\Vert M^\eps \right\Vert _{p-\rm{var};\left[ 0,t\right] } ^{\alpha},
\end{equation}
for every $\ell=1,\ldots ,k+1$, every $m\in \mathbf{K}_{\left(
i_{1},\ldots ,i_{k}\right) }$ and every $\left( i_{1},\ldots ,i_{k}\right) \in
\left\{ 1,\ldots ,d\right\} ^{k}$, where we have set $M^\eps=(X^\eps,\mathbf{J^\eps},(\mathbf{J}^\eps)^{-1})$.
\end{proposition}
\begin{remark}\label{rk: f_ml}
The $f^{\eps,m}_l$'s in the above proposition are constructed recursively from $\partial_x^kV(X^\eps_s)$, $\bj_s^\eps, (\bj_s^\eps)^{-1}$ and multiplications of them through iterated integrals with respect to $B_s$.  By the discussion in \cite[Chapter 11]{FV-bk}, components of $\partial^k_xV(X_s^\eps), \bj_s^\eps$ and $(\bj^\eps_s)^{-1}$ are smooth with respect to $\eps$ in $\mathcal{C}^{p-var}([0,1], \mr)$. Moreover,  derivatives of $\partial^k_xV(X_s^\eps), \bj_s^\eps$ and $(\bj^\eps_s)^{-1}$ (with respect to $\eps$) satisfy  linear equations similar to (\ref{eq:jacobian}) and by the same technique used in proving Proposition \ref{prop:moments-jacobian}, one can show that the derivatives are in $L^r(\mp)$ for all $r\geq 1$ and uniform in $\eps\in(0,1]$. Hence $f^{\eps,m}_l$ is differentiable with respect to $\eps$ in  $\mathcal{C}^{p-var}([0,1], \mr)$. Moreover for all $r\geq  1$, one has
$$\me\sup_{\eps\in(0,1]}\left\|\frac{d f^{\eps,m}_l}{d\eps}\right\|^r_{p-var; [0,1]}< \infty$$
for all $l$ and $m$.
\end{remark}

Let $\Phi: \msh\to \mathcal{C}([0,1],\mathbb{R}^{n})$ be given by solving the ordinary diferential equation
\begin{align}\label{phi}\Phi_t(h)=x+\sum_{i=1}^d\int_0^tV_i(\Phi_s(h))dh_s^i.\end{align}
Following  Proposition \ref{prop:imbed-bar-H}, Proposition \ref{th:express DX} and Remark \ref{rk: f_ml}, we have

\begin{proposition}\label{th: smoothness in eps}
For each $h\in\msh$, one has
\begin{align}\label{smoothness in eps}
\lim_{\eps\downarrow0}\frac{1}{\eps}\left(\Phi_1(\eps B+h)-\Phi_1(h)\right)=Z(h),
\end{align}
in the topology of $\md^\infty$.
\end{proposition}

\begin{proof}
We need to show that the convergence stated in the proposition takes place in $\|\cdot\|_{k,r}$ for any $k\geq 0$ and $r\geq 1$.  First note that $Y^\eps_t:=\Phi_{t}(\eps B+h)$ satisfies equation
\begin{align}
Y_t^\eps=x+\sum_{i=1}^d\int_0^tV_i(Y_s^\eps)d(\eps B_s^i+h^i).
\end{align}
When $H>\frac{1}{2}$, note that for any integer $k\geq0$ and fixed $s_1,...,s_k$,  $\bd^k_{s_1,...,s_k} Y_t^\eps$ satisfies an linear equation.  By the discussion in \cite[Chapter 11]{FV-bk},  $\bd^k_{s_1,...,s_k} Y_t^\eps$ is differentiable with respect to $\eps$ as a random vector. The fact that $\bd^k_{s_1,...,s_k} Y_t^\eps$ is also differentiable with respect to the norm $\me\|\cdot\|_{\mathcal{H}^{\otimes k}}$ follows from the fact that $\|\cdot\|_{\mathcal{H}^{\otimes k}}$ is controlled by the sup-norm and the integrability of the system.

Next, we focus on the case $H<\frac{1}{2}$, and divide the proof into three steps.

\noindent{\it Step 1}:  When $k=0$.  It is known that $Y^\eps$ is smooth (path-wise) with respect to $\eps$ in the topology of $\mathcal{C}^{p-var}([0,1],\mr^n)$ (cf. \cite[Chapter 11]{FV-bk}). To  see that the convergence in (\ref{smoothness in eps}) also take place in $L^r(\mp)$, one only needs to note that
$$\frac{Y_1^\eps-Y_1^0}{\eps}=\frac{1}{\eps}\int_0^\eps \frac{d Y_1^\theta}{d\theta}d\theta$$
where the above is considered to be an equation in $\mathcal{C}^{p-var}([0,1]; \mr^n)$.  Since $dY_1^\theta/d\theta$ satisfies a linear equation,  $\|dY_1^\theta/d\theta\|_{p-var; [0,1]}$ is uniformly integrable with respect to $\theta\in[0,1]$ in $L^r(\mp)$ for all $r\geq1$. Hence one can conclude that $|(Y_1^\eps-Y_1^0)/\eps|$ is uniformly integrable in $L^r(\mp)$. The claimed convergence follows for the case $k=0$.

\ \\
\noindent{\it Step 2}: When $k=1$.  Denote by $J^\eps_t=\frac{\partial Y_t^\eps}{\partial x}$, the Jacobean of $Y^\eps_t$. We have for any $h_1\in\msh$
$$\bd_{h_1}Y_1^\eps=\eps\int_0^1 J_1^\eps (J_s^{\eps})^{-1}V_i(Y^\eps_s)dh_{1}^i(s).$$
Let $f_s^{\eps,1}= \eps J_1^\eps (J_s^{\eps})^{-1}V_i(Y^\eps_s)$, the integrand in the above integral.  It can be shown that  $f_t^{\eps,1}$ is smooth with respect to $\eps$ in  $\mathcal{C}^{p-var}([0,1];\mr^n)$. In particular, we have
\begin{align*}
&\left|\frac{\bd_{h_1}Y_1^\eps-\bd_{h_1}Y_1^0}{\eps}-\frac{d \bd_{h_1}Y_1^\eps}{d\eps}\big|_{\eps=0}\right|\\
=&\quad\left|\int_0^1\frac{f^{\eps,1}_s-f^{0,1}_s}{\eps}-\frac{df^{\eps,1}_s}{d\eps}\big|_{\eps=0}dh_1^i(s) \right|\\
\leq& C\left(\left|\frac{f^{\eps,1}_0-f^{0,1}_0}{\eps}-\frac{df^{\eps,1}_0}{d\eps}\big|_{\eps=0}\right|+\left\|\frac{f^{\eps,1}_s-f^{0,1}_s}{\eps}-\frac{df^{\eps,1}_s}{d\eps}\big|_{\eps=0}\right\|_{p-var;[0,1]}\right)\|h_1\|_{q-var;[0,1]}\\
\leq& C\left(\left|\frac{f^{\eps,1}_0-f^{0,1}_0}{\eps}-\frac{df^{\eps,1}_0}{d\eps}\big|_{\eps=0}\right|+\left\|\frac{f^{\eps,1}_s-f^{0,1}_s}{\eps}-\frac{df^{\eps,1}_s}{d\eps}\big|_{\eps=0}\right\|_{p-var;[0,1]}\right)\|h_1\|_{\msh}\ \longrightarrow 0\quad\mathrm{as}\ \eps\downarrow 0.
\end{align*}
This implies that the convergence
$$\left|\frac{\bd_{h_1}Y_1^\eps-\bd_{h_1}Y_1^0}{\eps}-\frac{d \bd_{h_1}Y_1^\eps}{d\eps}\big|_{\eps=0}\right|\to 0\quad \quad \mathrm{as}\ \eps\downarrow 0,$$
is uniform in $h_1\in\msh$. Hence $$\frac{ \bd Y^\eps_1- \bd Y^0_1}{\eps} \to \frac{d \bd Y^\eps_1}{d\eps}\big|_{\eps=0}$$ in $\mathcal{{H}}$ almost surely. The fact that the convergence is also in $\|\cdot\|_{1,r}$ follows from the uniform $L^r(\mp)$ integrability of $\left\|\frac{ \bd Y^\eps_1- \bd Y^0_1}{\eps} - \frac{d \bd Y^\eps_1}{d\eps}|_{\eps=0}\right\|_{\mathcal{H}}$ in $\eps$.

\ \\
\noindent{\it Step 3}: Now we proof for general $k\geq 1$.  Note that for $ h_{1},\ldots,h_{k} \in \msh$, the directional derivative $\mathbf{D}_{h_{1}}\ldots \mathbf{D}_{h_{k}}Y^\eps_t $ exists for any $t\in \left[
0,1\right] .$ Moreover, by Proposition \ref{th:express DX} (with a slight modification) there exists a collection of finite indexing sets
\begin{equation*}
\left\{ \mathbf{K}_{\left( i_{1},\ldots ,i_{k}\right) }:\left(
i_{1},\ldots ,i_{k}\right) \in \left\{ 1,\ldots ,d\right\} ^{k}\right\},
\end{equation*}%
such that for every $j\in \left\{ 1,..,n\right\} $ we have
\begin{eqnarray}
&&\mathbf{D}_{h_{1}}\ldots \mathbf{D}_{h_{k}}Y_{1}^{\eps,j}  \label{high order rep} \\
&=&\sum_{i_{1},\ldots ,i_{k}=1}^{d}\sum_{m\in \mathbf{K}_{\left(
i_{1},\ldots ,i_{k}\right) }}\int_{0<t_{1}<\cdots <t_{k}<1}f_{1}^{\eps,m}\left(
t_{1}\right) \ldots f_{k}^{\eps,m}\left( t_{k}\right) f_{k+1}^{\eps,m}\left( 1\right)
d h_1^{i_{1}}(t_{1})\ldots d h^{i_{k}}_k(t_{k}),  \notag
\end{eqnarray}%
for some functions $f_{\ell}^{\eps,m}$ which are in $\cac^{p-{\rm var}}\left( \left[
0,1\right] ,%
\mathbb{R}
\right) $ for every $\ell$ and $m$. By Remark \ref{rk: f_ml} each $f^{\eps,m}_l$ is smooth with respect to $\eps$ in $\mathcal{C}^{p-var}([0,1]; \mr)$ with uniform integrable derivatives.  Now by a similar argument to that in {\it Step 2}  the proof is completed.
\end{proof}

Finally, we close the discussion of this section by the following large deviation principle that will be needed later.
\begin{theorem}\label{th: LDP}
Let $\Phi$ be given in (\ref{phi}), which is a differentiable mapping from $\msh$ to $\mathcal{C}([0,1],\mathbb{R}^{n})$. Denote by $\gamma_{\Phi_1(h)}$ the deterministic Malliavin matrix
 of $\Phi_1(h)$, i.e., $\gamma^{ij}_{\Phi_1(h)}=\langle \bd \Phi_1^i(h), \bd\Phi_1^j(h)\rangle_\ch$, and introduce the following functions on $\mr^n$ and $\mr^n\times\mr$, respectively
$$I(y)=\inf_{\Phi_1(h)=y}\frac{1}{2}\|h\|_{\msh}^2, \quad\textnormal{and}\ \  I_R(y,a)=\inf_{\Phi_1(h)=y, \gamma_{\Phi_1(h)}=a}\frac{1}{2}\|h\|_\msh^2.$$
Recall that $X_1^\eps$ is the solution to equation (\ref{equ: SDE}) and $\gamma_{X_1^\eps}$ is the Malliavin matrix of $X_1^\eps$. Then

(1) $X_1^\eps$ satisfies a large deviation principle with rate function $I(y)$.

(2)The couple $(X_1^\eps,\gamma_{X_1^\eps})$ satisfies a large deviation principle with rate function $I_R(y,a)$.

\end{theorem}
\begin{proof}
Fix any $p>\frac{1}{H}$. It is known (see \cite{FV-bk}) that $\mathbf{B}$ as a $G^{\lfloor p\rfloor}(\mr^d)$-valued rough path satisfies a large deviation principle in $p$-variation topology with good rate function given by
\begin{align*}
J(h)=\left\{\begin{array}{ll}\frac{1}{2}\|h\|^{2}_\msh\ \mathrm{if}\ h\in\msh\\ +\infty\quad \mathrm{otherwise}.
\end{array}
\right.
\end{align*}
It is clear $\Phi_1(\cdot): G^{\lfloor p\rfloor}(\mr^d) \to \mr^n$ is continues. 

Moreover, by a similar argument as in Proposition \ref{th: smoothness in eps}, one have for all $H>\frac{1}{4}$
$$\bd_\cdot \Phi_1(\cdot): G^{\lfloor p\rfloor}(\mr^d)\to\mathcal{H}$$ is continues.  
Hence $ \gamma_{\Phi_1(\cdot)}: G^{\lfloor p\rfloor}(\mr^d)\to \mr$ is continues for all $H>\frac{1}{4}$.

 Now note that $X_1^\eps=\Phi_1(\eps \mathbf{B})$ and $(X_1^\eps, \gamma_{X_1^\eps})=(\Phi(\eps\mathbf{B}), \gamma_{\Phi_1(\eps\mathbf{B})})$, the claimed result follows from the contraction principle.
\end{proof}

\section{Varadhan Estimates}

Recall that we are interested in a family of stochastic differential equations driven by fractional Brownian motions $B$ (with Hurst parameter $H>\frac{1}{4}$) of the following form
\begin{align*}
X_t^\eps=x+\eps\sum_{i=1}^d\int_0^tV_i(X_s^\eps)dB_s^i.
\end{align*}
We have defined the map $\Phi: \msh\to \mathcal{C}[0,1]$  by solving the ordinary deferential equation
$$\Phi_t(h)=x+\sum_{i=1}^d\int_0^tV_i(\Phi_s(h))dh_s^i.$$
Clearly, we have $X_t^\eps=\Phi_t(\eps \mathbf{B})$. Introduce the following functions on $\mr^n$, which depends on $\Phi$
$$d^2(y)=I(y)=\inf_{\Phi_1(h)=y}\frac{1}{2}\|h\|_{\msh}^2,\quad\mathrm{and}\quad d^2_R(y)=\inf_{\Phi_1(h)=y, \det\gamma_{\Phi_1(h)}>0}\frac{1}{2}\|h\|_\msh^2.$$

\bigskip
Assume Hypothesis \ref{hyp:elliptic} or Hypothesis \ref{UH condition} below, our  main result is the following.
\begin{theorem}\label{th: main result}
Let us denote by $p_\eps(y)$ the density of $X_1^\eps$. Then
\begin{align}\label{main claim 1}
\liminf_{\eps\downarrow0}\eps^2\log p_\eps(y)\geq -d^2_R(y),
\end{align}
and
\begin{align}\label{main claim 2}
\limsup_{\eps\downarrow0}\eps^2\log p_\eps(y)\leq -d^2(y).
\end{align}
Moreover, if $\inf_{\Phi_1(h)=y, \det\gamma_{\Phi_1(y)}>0}\det\gamma_{\Phi_1(h)}>0$, then
\begin{align}\label{main claim 3}
\lim_{\eps\downarrow0}\eps^2\log p_\eps(y)=-d^2_R(y).
\end{align}
\end{theorem}

A key ingredient in proving Theorem \ref{th: main result} is an estimate for the Malliavin derivative $\bd X_1^\eps$ and Malliavin matrix $\gamma_{X_1^\eps}$ of $X_1^\eps$. Since it is more involved when the vector fields $V_i$'s form a hypoelliptic system, we divide the estimates of theses two quantities into two parts: (1) when $V_i$'s are uniformly elliptic and (2) $V_i$'s are uniformly hypoelliptic (see Hypothesis \ref{UH condition}).


\subsection{Elliptic Case} Throughout our discussion in this subsection, we assume that $V_1,...,V_d$ form a uniformly elliptic system.
\begin{hypothesis} \label{hyp:elliptic}
The vector fields $V_1,\ldots ,V_d$ of equation \eqref{equ: SDE} is uniformly elliptic, that is
\begin{equation}\label{eq:hyp-elliptic}
 v^{*} V(x) V^{*}(x) v  \geq \lambda \vert v \vert^2, \qquad \text{for all } v,x \in \R^n,
\end{equation}
where we have set $V=(V_j^i)_{i=1,\ldots ,n; j=1,\ldots d}$ and where $\lambda$ designates  a positive constant.
\end{hypothesis}

\bigskip
Our main technical result in this subsection is the following. 

\begin{lemma}\label{th: Malliavin est}
Assume Hypothesis \ref{hyp:elliptic}.  For $H>\frac{1}{4}$, we have
\begin{itemize}
\item[(1)] $\sup_{\eps\in(0,1]}\|X_1^\eps\|_{k,r}<\infty$ for each $k\geq 1$ and $r\geq 1$.
\item[(2)] $\|\gamma_{X_1^\eps}^{-1}\|_r\leq c_r \eps^{-2}$ for any $r\geq 1$.
\end{itemize}
\end{lemma}
\begin{proof}
We follow the idea in \cite{BNOT}. First recall that the Malliavin derivative $\bd^i_t X_1^\eps$ can be expressed as $$\bd^i_t X_1^\eps  =\eps \mathbf{J^\eps}_{1 } (\mathbf{J}^{\eps}_{t})^{-1} V_i (X_t^\eps),$$ where $\mathbf{J^\eps}$ is the Jacobian process defined by $\mathbf{J}^\eps_t=\frac{\partial X^\eps_t}{\partial x}$. We divide the proof into two cases: $H>\frac{1}{2}$ and $\frac{1}{4}<H<\frac{1}{2}$.

When $H>\frac{1}{2}$, it is clear that (see \cite{HN}, for example)
$$\sup_{\eps\in(0,1]}\|X_1^\eps\|_{k,r}<\infty$$
for each $k\geq 1$ and $r\geq 1.$

In the following we prove the desired bounds for the Malliavin matrix.  Let
\[
\Gamma_\eps = \int_0^1 \int_0^1 (\mathbf{J^\eps}_{ v})^{-1} V(X_{v}^\eps)V(X_{u}^\eps)^*((\mathbf{J}_{ u}^\eps)^*)^{-1} |u-v|^{2H-2} du dv  .
\]
Our bound for the Malliavin matrix $\gamma_{X_1^\eps}$ is now reduced to prove that
\begin{equation}\label{eq:low-bnd-Sigma}
y^{*} \Gamma_\eps y \ge M_\eps \, |y|^{2},
\quad \text{for} \quad
y\in\R^{n},
\end{equation}
for a given random variable $M_\eps$ whose inverse admits moments of any order uniformly in $\eps\in[0,1]$. To this aim, notice first that
\begin{equation*}
y^{*} \Gamma_\eps y =
\int_0^1 \int_0^1
\lla f_{u} , \, f_{v} \rra_{\R^{d}}
|u-v|^{2H-2} \, du dv,
\quad \text{with} \quad
f_{u}\equiv V(X_{u}^\eps)^*((\mathbf{J}_{ u}^\eps)^{-1} )^* y.
\end{equation*}
Furthermore, thanks to the interpolation inequality  of \cite[Lemma 4.4]{BH} applied to $\gamma >H-\frac{1}{2}$, we have
\begin{equation}\label{eq:interpolation-H-L2}
\int_0^1 \int_0^1 \langle f_{u}, f_{v} \rangle |u-v|^{2H-2}dudv  \ge C \frac{\left( \int_0^1 v^\gamma (1-v)^\gamma | f_{v} |^2 dv\right)^2}{\| f \|^2_{\gamma}},
\end{equation}
where $\| f \|_{\gamma}$ is the $\gamma$-H\"older norm of $f$ on the interval $[0,1]$. As a consequence, since the ellipticity condition $| V(x) y |^2 \ge \lambda | y |^2$ holds true, it is readily checked that
\begin{equation}\label{eq:bnd-f}
|f_v|^{2} \ge \la \, |(\mathbf{J}_{ v}^\eps)^{-1} y |^{2}\ge \la \, \|\mathbf{J}_{ v}^\eps\|^{-2} | y |^{2},
\quad \text{and} \quad
\| f \|_{\gamma} \le c\, (1+\|X^\eps\|_{\gamma}) (1+\|\bj^{-1}\|_{\gamma}) | y |.
\end{equation}
Plugging these relations into \eqref{eq:interpolation-H-L2} we deduce that for every  $y \in \mathbb{R}^n$,
\[
 y^*\Gamma^{-1}_\eps y \le  c\,  (1+\|X^\eps\|_{\gamma})^{2} (1+\|(\bj^\eps)^{-1}\|_{\gamma})^{2} \| \bj^\eps\|^4_{\gamma}
\,  \vert y\vert^2,
\]
from which the desired result follows easily.

Next, we prove for the case $\frac{1}{4}<H<\frac{1}{2}$.   We first prove claim (1) for the Malliavin derivatives. Thanks to  Proposition \ref{th:express DX} and  Proposition \ref{prop:imbed-bar-H} we have, for $q>(H+\frac{1}{2})^{-1}$ and $h_1, \ldots, h_k \in \msh$ :
\begin{align*}
&|\bd_{h_1}\ldots \bd_{ h_k}X^\eps_1|\\
&\leq c_1 \Vert f_{k+1}^\eps \Vert_\infty(1+ \Vert f_1^\eps \Vert_{p-{\rm var};[0,t]}\ldots \Vert f_m^\eps \Vert_{p-{\rm var};[0,t]}) \Vert  h_1 \Vert_{q-{\rm var};[0,t]}\ldots \Vert  h_k \Vert_{q-{\rm var};[0,t]} \\
&\leq c_2(1+ \Vert M^\eps \Vert^{\alpha k}_{p-{\rm var};[0,t]}) \Vert h_1 \Vert_{q-{\rm var};[0,t]}\ldots \Vert  h_k \Vert_{q-{\rm var};[0,t]} \\
&\leq c_3 (1+\Vert M^\eps \Vert^{\alpha k}_{p-{\rm var};[0,t]}) \Vert  h_1 \Vert_{\msh}\ldots \Vert h_k \Vert_{\msh}.
\end{align*}
Hence for any $r\geq 1$, $$\me\|\bd^kX^\eps_1\|_{\ch^{\otimes k}}^r\leq c \left(1+\me\Vert M^\eps \Vert^{\alpha kr}_{p-{\rm var};[0,t]}\right)< \infty .$$

Next we prove the estimate for $\gamma_{X^\eps_1}$. 
Let $\mathcal{M}^{\eps,ij}_s= \langle \mathbf{D}_s X^{\eps,i}_1 , \mathbf{D}_sX^{\eps,j}_1  \rangle$. We can deduce that  for any $v\in\mr^n$,
\begin{align*}
v^{*}\gamma_{X_1^\eps}v=\sum_{i=1}^d\|v^{*}\bd_\cdot^iX_1^\eps\|_{\mathcal{H}}\geq c_H\int_0^1|v^{*}\bd_s^iX_1^\eps|^2ds=c_H\int_0^1v^{*}\mathcal{M}^\eps_svds.
\end{align*}
In the above, the inequality is obtained by the fact that $\mathcal{H}\subset L^2[0,1]$. Hence
\begin{align}\label{inequ gamma M}
(v^{*}\gamma_{X_1^\eps}v)^{-1} &\le \frac{1}{c_H  }\int_0^1(v^{*}\mathcal{M}^\eps_sv)^{-1} ds.
\end{align}
Let us now derive a suitable bound for $\cm^\eps$: recall that the Malliavin derivative $\bd^i_t X_1^\eps$ can be expressed as $\bd^i_t X_1^\eps  =\eps \mathbf{J^\eps}_{1 } (\mathbf{J}^{\eps}_{t})^{-1} V_i (X_t^\eps)$, where $\mathbf{J^\eps}$ is the Jacobian process defined by $\mathbf{J}^\eps_t=\frac{\partial X^\eps_t}{\partial x}$. We have
$$\mathcal{M}^{\eps,ij}_s=\langle \mathbf{D}_s X^{\eps,i}_1 , \mathbf{D}_sX^{\eps,j}_1  \rangle=\eps^2\langle (\mathbf{J}^\eps_1 (\mathbf{J}_s^\eps)^{-1}V(X_t^\eps))^i, (\mathbf{J}_1^\eps (\mathbf{J}^\eps_s)^{-1}V(X_t^\eps))^j\rangle.$$
Hence by the uniform integrability of $\mathbf{J}^\eps$ and $(\mathbf{J}^\eps)^{-1}$ in $\eps\in[0,1]$, and the uniform ellipticity of the vector fields $V_i$'s,  we easily bound $$\me\left[\sup_{\eps\in(0,1]}\sup_{s\in[0,1]}\left(\frac{\lambda^\eps_s}{\eps^2}\right)^{-r}\right] \le c_{r} $$ for any $r\geq1$, where $\lambda^\eps_s$ is the smallest eigenvalue of $\mathcal{M}_s^\eps$.
Hence  (\ref{inequ gamma M}) implies for any $r\geq 1$
\begin{align*}
\sup_{|v|=1}\mp\left\{\frac{v^{*}\gamma_{X_1^\eps}v}{\eps^2}\leq\delta\right\}&\leq \sup_{|v|=1} \mp\left\{\frac{\eps^2}{c_H}\int_0^1(v^{*}\mathcal{M^\eps}_sv)^{-1}ds\geq {\delta}^{-1}\right\}\\
&\leq\mp\left\{\sup_{s\in[0,1]}\left(c_H\frac{\lambda_s^\eps}{\eps^2}\right)^{-1}\geq{\delta}^{-1}\right\}\leq c_{r,H}\, \delta^r.
\end{align*}
Now we can conclude, by \cite[Lemma 2.3]{Nu06}, that $\|\gamma^{-1}_{X_1^\eps}\|_{r} \le c_{r} \, \eps^{-2}$.
This yields the claimed result.

\end{proof}

\bigskip
\subsection{Hypoelliptic case}
In this subsection, we extend the results in the above under a weaker assumption on the vector fields $V_1,\cdots, V_d$. We first introduce some notations. Let $\mathcal{A}=\{\emptyset\}\cup\bigcup^{\infty}_{k=1}\{1,2,\cdots,n\}^{k}$ and
$\mathcal{A}_{1}=A\setminus\{\emptyset\}$. We say that $I\in \mathcal{A}$ is a word of length $k$ if $I=(i_1,\cdots,i_k)$
and we write $|I|=k$. If $I=\emptyset$, then we denote $|I|=0$. For any integer $l\ge 1$, we denote by $\mathcal{A}(l)$
the set $\{I\in \mathcal{A}; |I|\le l\}$ and by $\mathcal{A}_{1}(l)$ the set $\{I\in \mathcal{A}_{1}; |I|\le l\}$ .
We also define an operation $\ast$ on $\mathcal{A}$ by
$I\ast J=(i_1,\cdots,i_k,j_1,\cdots,j_l)$ for $I=(i_1,\cdots,i_k)$ and $J=(j_1,\cdots, j_l)$ in $\mathcal{A}$.
We define vector fields $V_{[I]}$ inductively by
\[
V_{[j]}=V_{j}, \quad V_{[I\ast j]}=[V_{[I]}, V_{j}], \quad j=1,\cdots,d
\]

Now we introduce the following uniform hypoelliptic condition, which is in force through out the rest of the section.
\begin{hypothesis}\label{UH condition}(Uniform hypoelliptic condition)
The vector fields $V_{1},\cdots,V_{d}$ are in $C^{\infty}_{b}(\mathbb{R}^{n})$ and they form a uniform hypoelliptic system 
in the sense that there exist an integer $l$ and a constant $\lambda>0$ such that
\begin{align}\label{UH_condition}
\sum_{I\in \mathcal{A}_{1}(l)}\langle V_{[I]}(x), u\rangle^{2}_{\mathbb{R}^{n}}\ge \lambda \|u\|^{2}
\end{align}
holds for any $x,u\in \mathbb{R}^{n}$
\end{hypothesis}
\begin{remark}
It is clear that Hypothesis \ref{hyp:elliptic} is a special case of the above Hypothesis \ref{UH condition}.
\end{remark}

The main result of this subsection is the following counterpart of Lemma \ref{th: Malliavin est} in the hypoelliptic case.

\begin{lemma}\label{th: Malliavin est subelliptic}
Assume Hypothesis \ref{UH condition}.  For $H>\frac{1}{4}$, we have
\begin{itemize}
\item[(1)] $\sup_{\eps\in(0,1]}\|X_1^\eps\|_{k,r}<\infty$ for each $k\geq 1$ and $r\geq 1$.
\item[(2)] $\|\gamma_{X_1^\eps}^{-1}\|_r\leq c_r \eps^{-2l}$ for any $r\geq 1$.
\end{itemize}
\end{lemma}

\begin{remark}
It is clear that following the same lines in the proof of Lemma \ref{th: Malliavin est}, one has the claimed estimate (1) in Theorem \ref{th: Malliavin est subelliptic}. Hence in what follows, we focus on establishing (2) of Theorem \ref{th: Malliavin est subelliptic}. 
\end{remark}

\bigskip

Under the Hypothesis \ref{UH condition} above, for any $I\in \mathcal{A}_{1}$, we can find functions 
$\omega^{J}_{I}\in C^{\infty}_{b}(\mathbb{R}^{n},\mathbb{R})$ such that:

\begin{align}\label{bracket_omega_UH} 
V_{[I]}(x)=\sum_{J\in \mathcal{A}_{1}(l)}\omega^{J}_{I}(x)V_{[J]}(x)
\end{align}
holds for any $x\in \mathbb{R}^{n}$

We consider a family of SDEs indexed by $\epsilon\in(0,1]$:
\begin{align}\label{episilon_sde}
X^{\epsilon}_{t}=x+\epsilon \sum^{d}_{i=1}\int^{t}_{0}V_{i}(X^{\epsilon}_{t})dB^{i}_{s}
		=x+\sum^{d}_{i=1}\int^{t}_{0}V^{\epsilon}_{i}(X^{\epsilon}_{t})dB^{i}_{s},
\end{align}
where the rescaled vector fields $V^{\epsilon}_{i}$ are defined as $V^{\epsilon}_{i}(x)=\epsilon V_{i}(x)$. More generally, for any 
$I\in \mathcal{A}_{1}(l)$, we denote $V^{\epsilon}_{[I]}(x)=\epsilon^{|I|}V_{[I]}(x)$. Note that for $I\in \mathcal{A}_{1}(l)$, 
\begin{align*}
 V^{\epsilon}_{[I]}(x)=&\epsilon^{|I|}V_{I}(x)\\
                      =&\sum_{J\in \mathcal{A}_{1}(l)}\epsilon^{|I|}\omega^{J}_{I}(x)V_{J}(x)\\
		       =&\sum_{J\in \mathcal{A}_{1}(l)}\epsilon^{(|I|-|J|)}\omega^{J}_{I}(x)V^{\epsilon}_{[J]}(x)\\
		       =&\sum_{J\in \mathcal{A}_{1}(l)}\omega^{J,\epsilon}_{I}(x)V^{\epsilon}_{[J]}(x)  
\end{align*}
where $\omega^{J,\epsilon}_{I}(x)=\epsilon^{(|I|-|J|)}\omega^{J}_{I}(x)$.\\

It is known that for any $\epsilon \in(0,1]$ and 
any $t > 0$, the map $x\rightarrow X^{\epsilon}_{t}: \mathbb{R}^{n}\rightarrow \mathbb{R}^{n}$ is a flow of $C^{\infty}$ 
diffeomorphisms (see \cite{FV-bk}). We denote the Jacobian by $$\mathbf{J}_{t}^\eps=\partial_{x} X^\epsilon_t.$$ As being mentioned earlier,  $\bj_{t}^\eps$ and $(\bj_{t}^\eps)^{-1}$ satisfies
the following linear equations:
\[
\bj_{t}^\eps = \id_{d} +\sum_{j=1}^d \int_0^t DV^{\epsilon}_j (X^{\eps}_s) \, \bj_{s}^\eps \, dB^j_s,
\]
and its inverse $(J^{\epsilon}_{t})^{-1}$ satisfy the linear equation:
\[
(\bj_{t}^\eps)^{-1} = \id_{d} -\sum_{j=1}^d \int_0^t(\bj_{s}^\eps)^{-1} \, DV^{\epsilon}_j (X^{\eps}_s) \, dB^j_s,
\]

Let us introduce a linear system $\beta^{J,\epsilon}_{I}(t,x)$ that satisfies the following linear equations:
\begin{align}\label{beta}
\begin{cases}
   d\beta^{J,\epsilon}_{I}(t,x)=\displaystyle \sum^{d}_{i=1}\left(\sum_{K\in \mathcal{A}_{1}(l)}-\omega^{K,\epsilon}_{I\ast j}(X^{\epsilon}_{t})
   \beta^{J,\epsilon}_{K}(t,x)\right)dB^{i}_{t},\\
   \beta^{J,\epsilon}_{I}(0,x)=\delta^{J}_{I}.
\end{cases}
 \end{align}

\begin{lemma}\label{jacobian_inverse_beta}
Fix $\epsilon \in (0,1]$. For any $ I\in \mathcal{A}_{1}(l)$, we have:
\[
 (\bj_{t}^\eps)^{-1}(V^{\epsilon}_{[I]}(X^{\epsilon}_{t}))=\sum_{J\in \mathcal{A}_{1}(l)}\beta^{J,\epsilon}_{I}(t,x)V^{\epsilon}_{[J]}(x)
\]
\end{lemma}

\begin{proof}
To simpify the notation, let us denote 
\[
a^{\epsilon}_{I}(t,x)= (\bj_{t}^\eps)^{-1}(V^{\epsilon}_{[I]}(X^{\epsilon}_{t})),
\]
and 
\[
b^{\epsilon}_{I}(t,x)=\sum_{J\in \mathcal{A}_{1}(l)}\beta^{J,\epsilon}_{I}(t,x)V^{\epsilon}_{[J]}(x).
\]
Clearly by definition, we have $a^{\epsilon}_{I}(0,x)=b^{\epsilon}_{I}(0,x)=V^{\epsilon}_{[I]}(x)$. Next, we show that $a^{\epsilon}_{I}(t,x)$ and
$b^{\epsilon}(t,x)$ satisfy the same differential equation. Indeed, by change of variable formula, we have:
\begin{align*}
da^{\epsilon}_{I}(t,x)=&(\bj_{t}^\eps)^{-1}(V^{\epsilon}_{[I]}(X^{\epsilon}_{t}))\\
		      =&\sum^{d}_{j=1}(-1)(\bj_{t}^\eps)^{-1}[V^{\epsilon}_{[I]}, V^{\epsilon}_{j}](X^{\epsilon}_{t})(x)dB^{j}_{t}\\
		      =&\sum^{d}_{j=1}\sum_{J\in \mathcal{A}_{1}(l)}-\omega^{J,\epsilon}_{I\ast j}(X^{\epsilon}_{t})
		      (\bj_{t}^\eps)^{-1}V^{\epsilon}_{[J]}(X^{\epsilon}_{t})dB^{j}_{t}\\
		      =&\sum^{d}_{j=1}\sum_{J\in \mathcal{A}_{1}(l)}-\omega^{J,\epsilon}_{I\ast j}(X^{\epsilon}_{t})a^{\epsilon}_{J}(t,x)dB^{j}_{t}.
\end{align*}
On the other hand, by the definition of $\beta^{J, \epsilon}_{I}(t,x)$, we have:
\begin{align*}
db^{\epsilon}_{I}(t,x)=&d(\sum_{K\in \mathcal{A}_{1}(l)}\beta^{K,\epsilon}_{I}(t,x)V^{\epsilon}_{[K]}(x))\\
		       =&\sum_{K\in \mathcal{A}_{1}(l)}d\beta^{K,\epsilon}_{I}(t,x)V^{\epsilon}_{[K]}(x)\\
		       =&\sum_{j=1}^{d}\sum_{J\in \mathcal{A}_{1}(l)}-\omega^{J,\epsilon}_{I\ast j}(X^{\epsilon}_{t})
		       \sum_{K\in \mathcal{A}_{1}(l)}\beta^{K,\epsilon}_{J}(t,x)V^{\epsilon}_{[K]}(x)dB^{j}_{t}\\
		       =&\sum_{j=1}^{d}\sum_{J\in \mathcal{A}_{1}(l)}-\omega^{J,\epsilon}_{I\ast j}(X^{\epsilon}_{t})
		       b^{\epsilon}_{J}(t,x)dB^{j}_{t}.
\end{align*}
The result follows by the uniqueness of the differential equation.
\end{proof}

Let us introduce the following notation:  for any $I$, $J \in \mathcal{A}_{1}(l)$, we define
\[
 M^{\epsilon}_{I,J}(t,x)=t^{-(|I|+|J|)H}\langle \beta^{\epsilon, I}(\cdot,~x)1_{[0,t]}(\cdot), \beta^{\epsilon, J}(\cdot,~x)1_{[0,t]}(\cdot)\rangle_{\mathcal{H}}.
\]

In what follows, we will only consider the case $t=1$ and write $M^{\epsilon}_{I,J}(x)$ instead of $M^{\epsilon}_{I,J}(1,x)$. We cite the 
paper \cite[Theorem 3.5]{BOZ} for the following result:

\begin{proposition}\label{M inverse}
 For any $p \in (1, \infty)$, 
 \[
 \sup_{\epsilon\in (0,1], x\in \mathbb{R}^{n}}\mathbb{E}(\|(M^{\epsilon}_{I,J}(x))_{I,J\in \mathcal{A}_{1}(l)}\|^{-p})< \infty 
 \]
\end{proposition}

Finally, we are able to prove (2) of Lemma \ref{th: Malliavin est subelliptic}.

\begin{lemma}\label{upperbound inverse_UH}
Fix $\epsilon \in (0,1]$ and let $\gamma_{X^{\epsilon}_{1}}$ be the Malliavin matrix of $X^{\epsilon}_{1}$, then $\gamma_{X^{\epsilon}_{1}}$ is invertible and 
there exists a random variable $C\in \mathbb{L}^{p}$ for $p\ge 2$ such that
\[
\lambda_{\max}(\gamma^{-1}_{X^{\epsilon}_{1}})\leq \frac{C}{\epsilon^{2l}} \quad a.s.
\]
\end{lemma}

\begin{remark}
It follows from the above lemma that for any $r\geq 1$
$$\|\gamma_{X_1^\eps}^{-1}\|_r\leq c_r\eps^{-2l},\quad\quad \eps\in(0,1),$$
for some constant $c_r$ depending on $r$.
\end{remark}

\begin{proof}
First note that:
\begin{align*}
 D^{i}_{t}X^{\epsilon}_{1}=& \bj_{1}^{\eps}(\bj_{t}^\eps)^{-1}V^{\epsilon}_{i}(X^{\epsilon})\\
			   =&\sum_{J\in \mathcal{A}_{1}(l)}\beta^{J,\epsilon}_{i}(t,x)\bj_{1}^{\eps}V^{\epsilon}_{[J]}(x) 
\end{align*}
and 
\[
\gamma^{i,j}_{X^{\epsilon}_{1}}=\langle D_{t}X^{\epsilon, i}_{1}, D_{t}X^{\epsilon, j}_{1}\rangle_{\mathcal{H}}.
\]
Hence we have the following expression for $\gamma_{X^{\epsilon}_{1}}$:
\[
\gamma_{X^{\epsilon}_{1}}=\sum_{I,J\in \mathcal{A}_{1}(l)}\langle \beta^{I,\epsilon}(\cdot, x), 
\beta^{J,\epsilon}(\cdot, x)\rangle_{\mathcal{H}}\bj_{1}^{\eps}V^{\epsilon}_{[I]}(x)V^{\epsilon}_{[J]}(x)^{\ast}(\bj_{1}^{\eps})^{\ast}
\]
Now pick $u\in \mathbb{R}^{n}$, we have:
\begin{align*}
u^{\ast}\gamma_{X^{\epsilon}_{1}}u=&\sum_{I,J\in \mathcal{A}_{1}(l)}\langle \beta^{I,\epsilon}(\cdot, x), 
\beta^{J,\epsilon}(\cdot, x)\rangle_{\mathcal{H}}\langle V^{\epsilon}_{[I]}(x), (\bj_{1}^{\eps})^{\ast}u\rangle_{\mathbb{R}^{n}}
\langle V^{\epsilon}_{[J]}(x), (\bj_{1}^{\eps})^{\ast}u\rangle_{\mathbb{R}^{n}}\\
                                  =&\sum_{I,J\in \mathcal{A}_{1}(l)}M^{\epsilon}_{I,J}(x)
                                  \langle V^{\epsilon}_{[I]}(x), (\bj_{1}^{\eps})^{\ast}u\rangle_{\mathbb{R}^{n}}
                                  \langle V^{\epsilon}_{[J]}(x), (\bj_{1}^{\eps})^{\ast}u\rangle_{\mathbb{R}^{n}}\\
                                  \ge&\lambda_{\min}(M^{\epsilon}_{I,J}(x))\sum_{I\in \mathcal{A}_{1}(l)}
                                  \langle V^{\epsilon}_{[I]}(x), (\bj_{1}^{\eps})^{\ast}u\rangle^{2}_{\mathbb{R}^{n}}\\
                                  =&\lambda_{\min}(M^{\epsilon}_{I,J}(x))\sum_{I\in \mathcal{A}_{1}(l)}\epsilon^{2|I|}
                                  \langle V_{[I]}(x), (\bj_{1}^{\eps})^{\ast}u\rangle^{2}_{\mathbb{R}^{n}}\\
                                  \ge&\epsilon^{2l}\lambda_{\min}(M^{\epsilon}_{I,J}(x))\lambda\|(\bj_{1}^{\eps})^{\ast}u\|^{2}\\
                                  \ge&\lambda\epsilon^{2l}\lambda_{\min}(M^{\epsilon}_{I,J}(x))\lambda^{2}_{\min}(\bj_{1}^{\eps})\|u\|^{2}
\end{align*}
Now by choosing $u$ the eigenvector that corresponds to $\lambda_{\min}(\gamma_{X^{\epsilon}_{1}})$, we obtain:
\[
\lambda_{\min}(\gamma_{X^{\epsilon}_{1}})\ge\lambda\epsilon^{2l}\lambda_{\min}(M^{\epsilon}_{I,J}(x))\lambda^{2}_{\min}(\bj_{1}^{\eps})
\]
Now by the uniform integrability of $(\bj_{1}^{\eps})^{-1}$ in $\epsilon \in (0,1]$, and proposition(\ref{M inverse}), we obtain:
\[
\lambda_{\max}(\gamma^{-1}_{X^{\epsilon}_{1}})\le \lambda^{-1}\lambda_{\max}((M^{\epsilon}_{I,J}(X))^{-1})\lambda^{2}_{\max}((\bj_{1}^{\eps})^{-1})\epsilon^{-2l}
\]
and it completes the proof.
\end{proof}

\bigskip
\subsection{Proof of Theorem \ref{th: main result}}

Now we are in position to prove the main result of this paper.
\bigskip

\noindent {\bf Proof of (\ref{main claim 1})} Fix $y\in\mr^n$. We only need to show for  $d^2_R(y)<\infty$, since if $d^2_R(y)=\infty$ the statement is trivial. Fix any $\eta>0$ and let $h\in\msh$ be such that $\Phi_1(h)=y, \det_{\gamma_{\Phi}}(h)>0$, and $\|h\|^2_{\msh}\leq d^2_R(y)+\eta$. Let $f\in C_0^\infty(\mr^n).$ By Cameron-Martin theorem for fractional Brownian motions, we have
$$\me f(X^\eps_1)=e^{-\frac{\|h\|_\msh^2}{2\eps^2}}\me f(\Phi_1(\eps B+h))e^\frac{B(h)}{\eps}.$$
Consider a function $\chi\in C^\infty(\mr), 0\leq \chi\leq 1,$ such that $\chi(t)=0$ if $t\not\in[-2\eta, 2\eta]$, and $\chi(t)=1$ if $t\in[-\eta,\eta]$. Then, if $f\geq 0$, we have
$$\me f(X^\eps_1)\geq e^{-\frac{\|h\|_\msh+4\eta}{2\eps^2}}\me \chi(\eps B(h))f(\Phi_1(\eps B+h)).$$
Hence, we obtain
\begin{align}\label{lower bound claim1}\eps^2\log p_\eps(y)\geq -(\frac{1}{2}\|h\|_\msh^2+2\eta)+\eps^2\log\me\big(\chi(\eps B(h))\delta_y(\Phi_1(\eps B+h))\big).\end{align}

On the other hand, we have
$$\me\big(\chi(\eps B(h))\delta_y(\Phi_1(\eps B+h))\big)=\eps^{-n}\me\left(\chi(\eps B(h))\delta_0\left(\frac{\Phi_1(\eps B+h)-\Phi_1(h)}{\eps}\right)\right).$$
Note that
$$Z_1(h)=\lim_{\eps\downarrow 0}\frac{\Phi_1(\eps B+h)-\Phi_1(h)}{\eps}$$
is a $n$-dimensional random vector in the first Wiener chaos with variance $\gamma_{\Phi_1}(h)>0$. Hence $Z_1(h)$ is non-degenerate and by Proposition \ref{th: smoothness in eps}, we obtain
$$\lim_{\eps\downarrow0}\me\left(\chi(\eps B(h))\delta_0\left(\frac{\Phi_1(\eps B+h)-\Phi_1(h)}{\eps}\right)\right)=\me\delta_0(Z_1(h)).$$
Therefore,
$$\lim_{\eps\downarrow0}\eps^2\log\me\big(\chi(\eps B(h))\delta_y(\Phi_1(\eps B+h))\big)=0.$$
Letting $\eps\downarrow0$ in (\ref{lower bound claim1}) we obtain
$$\liminf_{\eps\downarrow0}\eps^2\log p_{\eps}(y)\geq-(\frac{1}{2}\|h\|^2_\msh+2\eta)\geq -(d^2_R(y)+3\eta).$$
Since $\eta>0$ is arbitrary, this completes the proof. \hfill$\Box$

As a direct consequence of the above proof, we have the following lower bound for the density function for small $\eps\in(0,1)$.

\begin{corollary}
Assume the same conditions as Theorem \ref{th: main result}. Denote by $p_\eps(y)$ the density of $X_1^\eps$. Then for all small $\eps\in (0,1)$ we have
$$p_\eps(y)\geq \frac{C}{\eps^n}e^{-\frac{d^2_R(y)}{2\eps^2}},$$for  some constant $C>0.$
\end{corollary}

\bigskip

\noindent{\bf Proof of (\ref{main claim 2})}. Fix a point $y\in\mr^n$ and consider a function $\chi\in C_0^\infty(\mr^n), 0\leq\chi\leq1$ such that $\chi$ is equal to one in a neighborhood of $y$. The density of $X_1^\eps$ at point $y$ is given by
$$p_\eps(y)=\me\chi(X_1^\eps)\delta_y(X_1^\eps).$$
By Proposition \ref{th: IBP}, we can write
\begin{align*}
\me\chi(X_1^\eps)\delta_y(X_1^\eps)=&\me\left(\mathbf{1}_{\{X_1^\eps>y\}}H_{(1,2,...,n)}(X_1^\eps,\chi(X_1^\eps))\right)\\
\leq&\me|H_{(1,2,...,n)}(X_1^\eps,\chi(X_1^\eps))|\\
=&\me\big(|H_{(1,2,...,n)}(X_1^\eps,\chi(X_1^\eps))|\mathbf{1}_{\{X_1^\eps\in \mathrm{supp}\chi\}}\big)\\
\leq&\mp(X_1^\eps\in\mathrm{supp}\chi)^\frac{1}{q}\|H_{(1,..,n)}(X_1^\eps,\chi(X_1^\eps))\|_p,
\end{align*}
where$\frac{1}{p}+\frac{1}{q}=1$. By Remark \ref{est H} we know that
$$\|H_{(1,...,n)}(X_1^\eps,\chi(X_1^\eps))\|_p\leq C_{p,q}\|\gamma_{X_1^\eps}^{-1}\|_\beta^m\|\bd X_1^\eps\|_{k,\gamma}^r\|\chi(X_1^\eps)\|_{k,q},$$
for some constants $\beta, \gamma>0$ and integers $k,m, r$. Thus, by Lemma \ref{th: Malliavin est} we have
$$\lim_{\eps\downarrow0}\eps^2\log\|H_{(1,...,n)}(X_1^\eps,\chi(X_1^\eps))\|_p=0.$$

Finally by Theorem \ref{th: LDP}, the large deviation principle for $X_1^\eps$ ensures that for small $\eps$ we have
$$\mp(X_1^\eps\in\mathrm{supp}\chi)^\frac{1}{q}\leq e^{-\frac{1}{q\eps^2}(\inf_{y\in\mathrm{supp}\chi}d^2(y))}$$
which concludes the proof.  \hfill$\Box$

\bigskip

\noindent{\bf Proof of (\ref{main claim 3})}.  Fix a point $y\in\mr^n$ and suppose that
$$\gamma:=\inf_{\Phi(h)=y, \det\gamma_{\Phi}(h)>0}\det\gamma_{\Phi}(h)>0.$$
Let  $\chi\in C_0^\infty(\mr^n), 0\leq\chi\leq1$ be a function such that $\chi$ is equal to one in a neighborhood of $y$,  and $g\in C^\infty(\mr), 0\leq g\leq1$, such that $g(u)=1$ if $|u|<\frac{1}{4}\gamma$, and $g(u)=0$ if $|u|>\frac{1}{2}\gamma$. Set $G_\eps=g(\det\gamma_{{X_1^\eps}})$. As before, we have
$$\me\chi(X_1^\eps)\delta_y(X_1^\eps)=\me G_\eps\chi(X_1^\eps)\delta_y(X_1^\eps)+\me(1-G_\eps)\chi(X_1^\eps)\delta_y(X_1^\eps)=I_1+I_2.$$
In what follows, we estimate $I_1$ and $I_2$ respectively.

\ \\
\noindent{ \it Estimate of $I_1$}: Let $\{\phi_n, n\geq 1\}$ be an orthonormal basis for $\mathcal{H}$. Let
$$B^N=\sum_{i=1}^N B(\phi_n)\phi_n$$
be the Karhunen-Loeve type approximation of $B$. Denote by $\mathbf{B}^N$ the lift of $B^N$ to $G^{\lfloor p\rfloor}(\mr^d)$. It has been shown (cf. \cite{FV-bk}) that $\mathbf{B}^N$ converges to $\mathbf{B}$ in the rough path topology in $L^r(\mp)$ for any $r\geq 1$. Then by the continuity of the It\^{o}'s map $\Phi$, we see immediately that $X_1^{\eps,N}=\Phi(\eps\mathbf{B}^N)$ converges to $X_1^{\eps}=\Phi(\eps\mathbf{B})$, as $N$ approaches to infinity,  in $L^r(\mp)$ for any $r\geq 1$.  Moreover, one can show that this convergence indeed takes place in $\md^\infty$, which can be seen by a similar (but simpler) argument to the proof of Lemma \ref{th: smoothness in eps}.

Now we can claim that $\me G_\eps\chi(X_1^\eps)\delta_y(X_1^\eps)=0$. Because, otherwise, based on the above Karhunen-Loeve type approximation and by some standard argument, one can find an element $\eps h\in\mathcal{H}$ such that $\Phi(\eps h)=y$ and $0<\det\gamma_\Phi(\eps h)<\frac{\gamma}{2},$ and this is in contradiction with the definition of $\gamma$ (see \cite[Proposition 4.2.1]{Nu-flour} for more details).

\ \\
\noindent{\it Estimate of $I_2$}: Proceding as in the proof of (\ref{main claim 2}) we obtain
\begin{align*}
\me(1-G_\eps)\chi(X_1^\eps)\delta_y(X_1^\eps)=&\me(\mathbf{1}_{\{X_1^\eps>y\}}H_{(1,...,n)}(X_1^\eps,(1-G_\eps)\chi(X_1^\eps)))\\
\leq&\me|H_{(1,...,n)}(X_1^\eps,(1-G_\eps)\chi(X_1^\eps))|\\
\leq&\me\big(|H_{(1,...,n)}(X_1^\eps\chi(X_1^\eps))|\mathbf{1}_{\{X_1^\eps\in\mathrm{supp}\chi, \det\gamma_{X_1^\eps}\geq\frac{1}{4}\gamma\}}\big)\\
\leq&\mp\left(X_1^\eps\in\mathrm{supp}\chi, \det\gamma_{X_1^\eps}\geq\frac{1}{4}\gamma\right)^{\frac{1}{q}}\|H_{(1,...,n)}(X_1^\eps,\chi(X_1^\eps))\|_p.
\end{align*}

Finally, by Lemma \ref{th: Malliavin est} and large deviation principle stated in Theorem \ref{th: LDP} for the couple $(X_1^\eps, \gamma_{X_1^\eps})$, we have for any $q>1$
\begin{align*}
\limsup_{\eps\downarrow0}\eps^2\log p_\eps(y)\leq&-\frac{1}{2q}\inf_{\Phi(h)\in\mathrm{supp}\chi, \det\gamma_{\Phi}(h)\geq\frac{1}{4}\gamma}\|h\|_\mathcal{H}^2\\
\leq&-\frac{1}{2q}\inf_{y\in\mathrm{supp}\chi}d^2_R(y).
\end{align*}
The proof is completed.\hfill$\Box$

\end{document}